\documentclass[11pt]{article}
\usepackage{fullpage}
\usepackage{graphicx}
\usepackage{amssymb}
\usepackage{amsmath}
\usepackage{amsthm}

\usepackage{stmaryrd}
\usepackage{mathabx}
\usepackage{mathrsfs}

\newtheorem{theorem}{Theorem}[section]
\newtheorem{lemma}[theorem]{Lemma}

\theoremstyle{definition}

\newcommand{\ignore}[1]{}
             
\newcommand{\hcm}[1][1]{\hspace*{#1 cm}}
\newcommand{\rb}[2]{\raisebox{#1 mm}[0mm][0mm]{#2}}
\newcommand{\istrut}[2][0]{\rule[- #1 mm]{0mm}{#1 mm}\rule{0mm}{#2 mm}}
\newcommand{\zero}[1]{\makebox[0mm][l]{$#1$}}

\newcommand{\angbrack}[1]{\left< #1 \right>}
\newcommand{\abs}[1]{\left| #1 \right|}

\newcommand{\Renyi}{R\'{e}nyi}

\newcommand{\Furedi}{F\"{u}redi}

\newcommand{\Korandi}{Kor\'{a}ndi}
\newcommand{\Kovari}{K\H{o}v\'{a}ri}
\newcommand{\Sos}{S\'{o}s}
\newcommand{\Turan}{Tur\'{a}n}
\newcommand{\Gyori}{Gy\H{o}ri}
\newcommand{\Ex}{\operatorname{Ex}}
\newcommand{\bu}{\bullet}

\newcommand{\vect}{\mathbf{v}}
\newcommand{\I}{\mathcal{I}}
\renewcommand{\S}{\mathcal{S}}
\newcommand{\type}{\mathsf{type}}

\title{On the Extremal Functions of Acyclic Forbidden 0--1 Matrices\thanks{S. Pettie is supported by NSF Grant CCF-2221980.  G. Tardos is supported by the National Research, Development and Innovation Office projects K-132696 and SNN-135643 and by the 
ERC Advanced Grants ``ERMiD'' and ``GeoScape.''
This research was performed 
while S. Pettie was visiting \Renyi{} Institute, Budapest.}}
\author{Seth Pettie\\ University of Michigan 
\and
G\'abor Tardos\\
\Renyi{} Institute
}
\date{}

\begin{document}

\maketitle

\begin{abstract}
The extremal theory of forbidden 0--1 matrices 
studies the asymptotic growth of the 
function $\Ex(P,n)$, which is 
the maximum weight of a matrix 
$A\in\{0,1\}^{n\times n}$ whose submatrices avoid 
a fixed pattern $P\in\{0,1\}^{k\times l}$.
This theory has been wildly successful at resolving problems
in
combinatorics~\cite{Klazar00,MarcusT04,CibulkaK12},
discrete and computational geometry~\cite{Furedi90,Aggarwal15,EfratS96,PachS91,Mitchell92,BienstockG91},
structural graph theory~\cite{GuillemotM14,BonnetGKTW21,BonnetKTW22}
and the analysis of data structures~\cite{Pettie10a,KozmaS20},
particularly corollaries of the dynamic optimality conjecture~\cite{ChalermsookGKMS15,ChalermsookG0MS15,ChalermsookGJAP23}.

All these applications use \emph{acyclic} patterns, meaning that
when $P$ is regarded as the adjacency matrix of a bipartite graph, the graph is acyclic.  The biggest open problem in this area is to bound $\Ex(P,n)$ for acyclic $P$.  Prior results~\cite{Pettie-FH11,ParkS13} have only ruled out the strict $O(n\log n)$ bound conjectured by \Furedi{} and Hajnal~\cite{FurediH92}.  It is consistent with prior results that $\forall P. \Ex(P,n)\leq n\log^{1+o(1)} n$, 
and also consistent that $\forall \epsilon>0.\exists P. \Ex(P,n) \geq  n^{2-\epsilon}$.

In this paper we establish a stronger lower bound on the extremal functions of acyclic $P$.  
Specifically, 
we give a new construction of relatively dense 0--1 matrices
with $\Theta(n(\log n/\log\log n)^t)$ 1s 
that avoid an acyclic $X_t$.  Pach and Tardos~\cite{PachTardos06} have conjectured that this type of result is the best possible, i.e., no acyclic $P$ exists for which $\Ex(P,n)\geq n(\log n)^{\omega(1)}$.
\end{abstract}

\section{Introduction}

The theory of forbidden 0--1 matrices subsumes or generalizes
many problems in extremal combinatorics, 
such as Davenport-Schinzel sequences~\cite{HS86,ASS89,Nivasch10,Pettie-DS-JACM,WellmanP18} and their generalizations~\cite{Pettie-GenDS11,Pettie15-SIDMA,FurediH92}, 
Zarankiewicz's problem~\cite{KovariST54}, 
and bipartite \Turan-type subgraph avoidance.
Forbidden 0--1 matrices have been applied to problems in 
discrete and computational geometry, the amortized analysis of data structures,
and in other areas of extremal combinatorics.  
Some highlights in geometry include bounding the number of unit-distances 
in a convex point set~\cite{Furedi90,Aggarwal15}, 
the number of critical placements of an $n$-gon 
in a hippodrome~\cite{EfratS96}, an analysis 
of the Bentley-Ottman line sweeping algorithm~\cite{PachS91},
and an analysis of Mitchel's algorithm for obstacle-avoiding
shortest paths in the plane~\cite{Mitchell92,BienstockG91}.
In data structures, forbidden 0--1 matrices have been 
used to analyze data structures based on binary search trees and path-compression~\cite{Pettie10a}, and more recently, to several 
corollaries of Sleator and Tarjan's~\cite{ST85} \emph{dynamic optimality} conjecture~\cite{ChalermsookGKMS15,ChalermsookG0MS15,ChalermsookGJAP23,KozmaS20}.
The most well-known application of forbidden 0--1 matrices 
is probably Marcus and Tardos's proof~\cite{MarcusT04} 
of the Stanley-Wilf conjecture, via Klazar's reduction~\cite{Klazar00}
to a \Furedi-Hajnal conjecture~\cite{FurediH92}.
They have also been used to bound Stanley-Wilf limits~\cite{Cibulka09,Fox13,CibulkaK17}, 
and to bound the size of sets of permutations 
with some fixed VC-dimension~\cite{CibulkaK12}. 
Most recently, results on searching for forbidden patterns~\cite{GuillemotM14}
inspired the definition of \emph{twin-width} for graphs and other binary structures~\cite{BonnetKTW22,BonnetGKTW21}.

\medskip
If $P\in \{0,1\}^{k\times l}$ and $A\in\{0,1\}^{n\times n}$, 
we say \emph{$A$ contains $P$}, written $P\prec A$, 
if there are rows $r_1<\cdots<r_k$ and 
columns $c_1<\cdots<c_l$ such that $P(i,j)=1 \rightarrow A(r_i,c_j)=1$.
If $P\nprec A$ then $A$ \emph{avoids $P$} or is \emph{$P$-free}.
The basic extremal function is defined as follows
\[
\Ex(P,n) = \max\{\|A\|_1 \mid A\in\{0,1\}^{n\times n} \text{ and } P\nprec A\},
\]
and may be generalized in various ways, e.g., 
to avoid \emph{sets} of forbidden patterns~\cite{Tardos05} or rectangular $A$~\cite{FurediH92,Pettie-GenDS11}, or
to $d$-dimensional matrices/patterns~\cite{KlazarM06,MethukuP17,Geneson19}.
Observe that $P$ and $A$ can be viewed
as incidence matrices of bipartite graphs, where
the two vertex sets are implicitly \emph{ordered}; 
let $G(P)$ be the \emph{unordered} undirected 
graph corresponding to $P$.

\Furedi{} and Hajnal~\cite{FurediH92} attempted to 
systematically classify small forbidden patterns by 
their extremal functions, and managed to do this 
for most weight-4 patterns, with the last holdouts being classified by Tardos~\cite{Tardos05}.  
For any pattern $P$, the extremal function $\Ex(P,n)$ 
is at least as large as the unordered (\Turan) extremal function for $G(P)$.  A natural question is to determine when they are the same, asymptotically, and the maximum factor by which they can differ.

\Furedi{} and Hajnal~\cite{FurediH92} 
concluded their article with several influential 
conjectures.  They first conjectured that when $P$ is a $k\times k$ permutation matrix, that $\Ex(P,n)\leq c_k n$, i.e., it is asymptotically the same as the \Turan{} number of $G(P)$.
Klazar~\cite{Klazar00} proved that this conjecture implies the Stanley-Wilf conjecture, and Marcus and Tardos~\cite{MarcusT04} proved both conjectures.  See~\cite{Geneson09,Cibulka09,CibulkaK17,Fox13,KlazarM06} for generalizations and sharper analyses of the leading constants.
\Furedi{} and Hajnal next conjectured that $\Ex(P,n)$ 
would never be more than a $\log n$-factor larger than 
the (unordered) \Turan{} number of $G(P)$.  
Perhaps doubting this conjecture, they immediately asked whether it held for \emph{acyclic} patterns $P$, i.e., if $G(P)$ is a forest,
is $\Ex(P,n)=O(n\log n)$?

Pach and Tardos~\cite{PachTardos06} refuted the second \Furedi-Hajnal conjecture, by exhibiting arbitrarily large $P$ for which $G(P)=C_{2k}$
is a $2k$-cycle but $\Ex(P,n)=\Omega(n^{4/3})$.  This implies the gap between the ordered and unordered extremal functions is $n^{1/3-\epsilon}$, where $\epsilon=1/k$ can be made arbitrarily small.
However, this refutation did not imply anything about the gap for \emph{acyclic} matrices.  Understanding acyclic patterns is important, as \emph{every} application to geometry, data structures, and combinatorics 
mentioned in the first paragraph uses only acyclic patterns.
Pettie~\cite{Pettie-FH11} disproved the last \Furedi-Hajnal conjecture by exhibiting a specific acyclic pattern $X$ for which $\Ex(X,n)=\Omega(n\log n\log\log n)$.  An unpublished manuscript of Park and Shi~\cite{ParkS13} extended this lower bound to a set of patterns $\{X_m\}$ for which 
$\Ex(X_m,n)=\Omega(n\log n\log\log n\log\log\log n\cdots \log^{(m)} n)$.

The constructions of~\cite{Pettie-FH11,ParkS13}
refuted the \emph{letter} of \Furedi{} and Hajnal's 
conjecture, but certainly not its \emph{spirit}.
Consider several non-trivial 
possibilities for the extremal function of an acyclic $P$.
\begin{description}
    \item[Absolute Polylog$(n)$.] There is an absolute constant $c\geq 1$ such that for any acyclic $P$, $\Ex(P,n)\leq n\log^{c+o(1)}n$.
    \item[Variable Polylog$(n)$.] For any acyclic $P$, there is a constant $c=c(P)$ such that $\Ex(P,n)\leq n\log^{c} n$.
    \item[Subpolynomial.] For every acyclic $P$, there is some $\epsilon(n)=o(1)$ depending on $P$ such that 
    $\Ex(P,n)\leq n^{1+\epsilon(n)}$.
    \item[Polynomial.] For some $c<2$, every acyclic $P$ has $\Ex(P,n)\leq O(n^c)$.
\end{description}
None of these upper bounds have been established or ruled out.
In particular, prior work~\cite{Pettie-FH11,ParkS13} 
does not preclude the possibility that 
{\bfseries Absolute Polylog$(n)$} holds even with $c=1$, 
and it is also possible the {\bfseries Polynomial} fails, 
i.e., for every $\epsilon>0$, there exists an acyclic $P$ for which $\Ex(P,n) = \Omega(n^{2-\epsilon})$.
Pach and Tardos~\cite{PachTardos06} conjectured broadly
that the {\bfseries Variable Polylog$(n)$} upper bound is true,
and conjectured more specifically that 
$\Ex(P,n)=O(n\log^{\|P\|_1-3} n)$.

\bigskip 

The biggest open problem in the theory of forbidden 0--1 matrices is
to understand \emph{acyclic patterns}.  On the upper bound side, we have
a perfect classification of all patterns with four 1s~\cite{FurediH92,Tardos05}, and a good classification for those with five 1s~\cite{PachTardos06}, up to a $\log n$ factor. 
For example, $\Ex(R_1,n)$ and $\Ex(R_2,n)$ are known to be 
$\Omega(n\log n)$ and $O(n\log^2 n)$~\cite{PachTardos06}.
\begin{align*}
R_1&=\left(\begin{array}{ccc}
\bu & \bu & \\
     &    & \bu\\
\bu &   & \bu\end{array}\right)
&R_2&=\left(\begin{array}{cccc}
\bu & \bu &    & \bu\\
\bu &    & \bu
\end{array}\right)
\end{align*}
\Korandi, Tardos, Tomon, and Weidert~\cite{KorandiTTW19}
defined a pattern $P$ to be \emph{class-$s$ degenerate}
if it can be written $P=\left(\begin{array}{c}P'\\P''\end{array}\right)$, where at most one column has a non-zero intersection with both $P'$ and $P''$,
and $P',P''$ are class-$(s-1)$ degenerate. 
(In the diagrams of $S_1,S_2$, a valid row partition cuts
at most one edge.)
Any $P$ with a single row is class-0 degenerate.
They proved that every class-$s$ degenerate $P$ has
\[
\Ex(P,n) \leq n\cdot 2^{O(\log^{1-\frac{1}{s+1}} n)}=n^{1+o(1)}.
\]
For example, $\Ex(S_1,n),\Ex(S_2,n)\leq n\cdot 2^{O(\log^{2/3} n)}$ as $S_1,S_2$ are class-2 degenerate.
\begin{align*}
    S_1&=\left(\begin{array}{cccc}
\bu\zero{\hcm[.1]\rb{0.9}{\rule{.8cm}{0.25mm}}} &   & \bu & \\
\zero{\hcm[.16]\rb{2.1}{\rotatebox{90}{\rule{.23cm}{0.25mm}}}}\bu\zero{\hcm[0]\rb{0.9}{\rule{1.35cm}{0.25mm}}} &   &      & \bu\\
    & \bu\zero{\hcm[.1]\rb{0.9}{\rule{.7cm}{0.25mm}}} &     & \zero{\hcm[.08]\rb{2.1}{\rotatebox{90}{\rule{.22cm}{0.25mm}}}}\bu\end{array}\right)
&S_2&=\left(\begin{array}{cccc}
    & \bu\zero{\hcm[.04]\rb{1}{\rule{.8cm}{0.25mm}}} &      & \bu\\
\bu\zero{\hcm[.1]\rb{0.9}{\rule{.8cm}{0.25mm}}} &     & \bu  &\\
\zero{\hcm[.16]\rb{2.1}{\rotatebox{90}{\rule{.23cm}{0.25mm}}}}\bu\zero{\hcm[.01]\rb{1}{\rule{1.35cm}{0.25mm}}} &     &      & \zero{\hcm[.07]\rb{2.1}{\rotatebox{90}{\rule{.7cm}{0.25mm}}}}\bu\end{array}\right)
\end{align*}
Clearly, a pattern and its transpose have the same extremal function. The smallest non-degenerate acyclic pattern whose transpose is also non-degenerate is 
the ``pretzel'' $T$; we have no non-trivial upper bounds
on $\Ex(T,n)$.
\[
T =\left(\begin{array}{cccc}
\bu\zero{\hcm[.05]\rb{0.9}{\rule{1.55cm}{0.25mm}}} &   &      & \bu\\
    & \bu&      &\\
\zero{\hcm[.15]\rb{2.1}{\rotatebox{90}{\rule{.7cm}{0.25mm}}}}\bu\zero{\hcm[0]\rb{.9}{\rule{1cm}{0.25mm}}} &   & \bu   &\\
    & \zero{\hcm[.15]\rb{2.1}{\rotatebox{90}{\rule{.7cm}{0.25mm}}}}\bu\zero{\hcm[0]\rb{.9}{\rule{.8cm}{0.25mm}}} &     & \zero{\hcm[.07]\rb{2.3}{\rotatebox{90}{\rule{1.15cm}{0.25mm}}}}\bu\end{array}\right)
\]

\subsection{New Result}

Our main result is a \emph{proper} refutation of 
the \Furedi-Hajnal conjecture for acyclic matrices
that rules out the {\bfseries Absolute Polylog$(n)$} world.
For any $t\ge2$, we give a new construction of 0--1 matrices containing
$\Omega(n(\log n/\log\log n)^t)$ 1s, and prove that 
they avoid a particular $2t\times(2t+1)$ 
acyclic pattern $X_t$ with $4t$ 1s.
\[
X_t = \left(\begin{array}{ccccccccc}
&\bu &&\bu&\cdots&\bu&&\bu&\bu\\
&&&&&&&&\bu\\
&\bu&&&&&&&\\
&&&&&&&&\bu\\
&\bu&&&&&&&\vdots\\
&\vdots&&&&&&&\bu\\
&\bu&&&&&&&\\
\bu&&\bu&\cdots&\bu&&\bu&&\bu
\end{array}\right)
\]
\begin{theorem}\label{thm:main}
For every $t\geq 2$, there exists a $2t\times(2t+1)$ 
acyclic pattern $X_t$ such that
\[
\Ex(X_t,n) = \left\{\begin{array}{l}
\Omega(n(\log n/\log\log n)^t),\\
O(n\log^{4t-3} n).
\end{array}\right.
\]
\end{theorem}

\subsection{Organization}

Section~\ref{sect:construction} presents the construction 
of 0--1 matrices and some of its properties unrelated 
to forbidden substructures.
Section~\ref{sect:forbidden-substructures} analyzes
the forbidden substructures, culminating in a proof of Theorem~\ref{thm:main}.
We conclude in Section~\ref{sect:conclusion} 
with a concise survey of open problems in 
0--1 matrices and related problems in \emph{ordered graphs}.

\section{A Construction of 0--1 Matrices}\label{sect:construction}

Define $[k]=\{1,\ldots,k\}$.
Fix a constant $t\geq 2$.
The rows and columns of $A_t$ are indexed
by length-$(tk)$ strings in the set
\[
\I = [k^t]^{tk}.
\]
An element of $\I$ is partitioned into 
$t$ \emph{blocks}, each of length $k$.  
If $a\in \I$,
let $a(p)\in [k^{t}]^k$ be its $p$th block, and
$a(p,q)\in [k^{t}]$ be the $q$th coordinate in block $p$.
We use angular brackets to denote any 
injective mapping from $[k]^r$ to $[k^r]$,
e.g.,  $\angbrack{j_1,j_2,j_3} = (j_1-1)k^2+(j_2-1)k+j_3$.
For $(j_1,\dots,j_t)\in[k]^t$, define $\vect = \vect[j_1,\ldots,j_t] \in\I$
to be the vector that is 0 in all coordinates except:
\begin{align*}
\vect({1,j_1}) &= \angbrack{\ } = 1,\\
\vect({2,j_2}) &= \angbrack{j_1} = j_1,\\
\hcm[.5]\cdots\hcm[.5] \\
\vect({r,j_r}) &= \angbrack{j_1,\ldots,j_{r-1}}\\
\hcm[.5]\cdots\hcm[.5]\\
\vect({t,j_t}) &= \angbrack{j_1,\ldots,j_{t-1}}.
\end{align*}

Define $\S$ to be the set of eligible vectors,
\[
\S=\{\vect[j_1,\dots,j_t] \mid (j_1,\dots,j_t)\in[k]^t\}.
\]
Letting $n=\abs{\I}$, $A_t$ is an $n\times n$ 0--1 matrix whose row and column sets are both indexed by $\I$, ordered lexicographically.
It is defined as follows:
\[
A_t(a,b) = \left\{\begin{array}{l@{\hcm}l}
1           & \mbox{if $b-a\in \S$,}\\
0           & \mbox{otherwise.}\\
\end{array}\right.
\]

\begin{lemma}\label{lem:density}
$\|A_t\|_1 = \Theta(n(\log n/\log\log n)^t)$.
\end{lemma}

\begin{proof}
$A_t$ is an $n\times n$ 0--1 matrix, where $n=k^{t^2k}=|\I|$.
Pick a uniformly random row $a\in \I$, and a uniformly 
random vector $\vect=\vect[i_1,\ldots,i_t] \in \S$.  
The probability that $a+\vect\in \I$ is legal
is the probability that for all $r\in[t]$, 
$a(r,i_r) + \angbrack{i_1,\ldots,i_{r-1}} \leq k^t$, 
which is at least $1-k^{-t+(r-1)}$ 
since $\angbrack{i_1,\ldots,i_{r-1}}\leq k^{r-1}$.
We have
\[
\Pr(a+\vect\in\I)\geq \prod_{r=1}^{t}(1-k^{-t+(r-1)}) 
\ge 1 -\sum_{r=1}^t k^{-r} > 1 - (k-1)^{-1}.
\]
Therefore the number of 1s in $A_t$ is at least
$(1-(k-1)^{-1})nk^t = \Theta(n(\log n/\log\log n)^t)$.
\end{proof}

\section{Forbidden Substructures}\label{sect:forbidden-substructures}

If $a,b\in \I$ are distinct vectors, their \emph{type} 
is the first block where they differ, i.e., 
\[
\type(a,b) = \min\{r \mid a(r)\neq b(r)\}.
\]

\begin{lemma}\label{lem:properties}
\begin{enumerate}
\item
For $a<b<c$, $a,b,c\in\I$, $\type(a,c)\leq \type(b,c)$.

\item  
    Suppose $A_t(a,c)=A_t(b,c)=1$, with $a<b$. 
    Let $c-a=\vect[i_1,\ldots,i_t]$ 
    and $c-b=\vect[j_1,\ldots,j_t]$. 
    If $\type(a,b)=r$, then  $i_q=j_q$ for $q<r$, $i_r<j_r$, 
    and the first coordinate where $a$ and $b$ differ is $(r,i_r)$.
\[
\begin{array}{c}
a\\
b
\end{array}
\left(\begin{array}{c}
\zero{\rb{3.5}{c}}\bu\\
\bu
\end{array}\right)
\istrut{7}
\]

\item
Suppose $A_t(a,c)=A_t(a,d)=1$, with $c<d$. Let $c-a=\vect[i_1,\ldots,i_t]$ and $d-a=\vect[j_1,\ldots,j_t]$. If $\type(c,d)=r$, then 
$i_q=j_q$ for $q<r$, $i_r>j_r$, 
and the first coordinate where $c$ and $d$ differ is $(r,j_r)$.
\[
\begin{array}{c}
a
\end{array}
\left(\begin{array}{cc}
\zero{\rb{3.5}{c}}\bu&\zero{\rb{3.5}{d}}\bu
\end{array}\right)
\istrut{7}
\]

\item
Suppose $A_t(b,c_1)=A_t(a,c_2)=A_t(b,d)=A_t(a,d)=1$,
where $a<b$ and $c_1<c_2<d$.
Then it is not possible that $\type(a,b)=\type(c_1,d)=\type(c_2,d)$.
\[
\begin{array}{c}
a\\
b
\end{array}
\left(\begin{array}{ccc}
\zero{\hcm[-.1]\rb{3.5}{$c_1$}}&\zero{\rb{3.5}{$c_2$}}\bu&\zero{\rb{3.5}{d}}\bu\\
\bu&&\bu
\end{array}\right)
\istrut{7}
\]

\item
Suppose $A_t(a,c_0)=A_t(b,c_1)=A_t(a,c_2)=A_t(b,d)=A_t(a,d)=1$,
where $a<b$ and $c_0<c_1<c_2<d$.
If $\type(a,b)\le\type(c_0,d)$, 
then $\type(c_0,d)<\type(c_2,d)$. 
\[
\begin{array}{c}
a\\
b
\end{array}
\left(\begin{array}{cccc}
\zero{\rb{3.5}{$c_0$}}\bu&\zero{\hcm[-.1]\rb{3.5}{$c_1$}}&\zero{\rb{3.5}{$c_2$}}\bu&\zero{\rb{3.5}{d}}\bu\\
&\bu&&\bu
\end{array}\right)
\istrut{7}
\]
\end{enumerate}
\end{lemma}

\begin{proof}
\underline{Part 1}
holds because $\I$ is ordered lexicographically.

\underline{Part 2.}
Let $s$ be the first index where $i_s\ne j_s$. Clearly, the first two coordinates where $\vect[i_1,\dots,i_t]$ and $\vect[j_1,\dots,j_t]$ differ is $(s,i_s)$ and $(s,j_s)$. This makes $a=c-\vect[i_1,\dots,i_t]$ and $b=c-\vect[j_1,\dots,j_t]$ also differ first at the same two coordinates, making $\type(a,b)=s$, so we have $r=s$. At coordinate $(s,j_s)$, $\vect[j_1,\dots,j_t]$ is $\langle j_1,\dots,j_{s-1}\rangle>0$ but $\vect[i_1,\dots,i_t]$ is zero there, so we have $a(s,j_s)>b(s,j_s)$. As $a<b$, $(s,j_s)$ cannot be the first coordinate where $a$ and $b$ differ, so we conclude that $i_s<j_s$ and $(s,i_s)=(r,i_r)$ is the coordinate of first difference between $a$ and $b$.

\underline{Part 3} can be proved the same way as Part~2.

\underline{Part 4.} 
Suppose, for the purpose of obtaining a contradiction,
that $\type(a,b)=\type(c_1,d)=\type(c_2,d)=r$. 
Both vectors $d-a$ and $d-b$ are in $\S$, so let
$$d-a=\vect[i_1,\dots,i_t],$$
$$d-b=\vect[j_1,\dots,j_t].$$
Applying Part~3 to row $b$, we know 
$c_1$ and $d$ first differ at coordinate $(r,j_r)$,
and applying Part~3 to row $a$, we know
$c_2$ and $d$ first differ at coordinate $(r,i_r)$. 
Applying Part~2 to column $d$, 
we have $i_r<j_r$.  However, since $c_2<d$ we have
\[
c_1(r,i_r) = d(r,i_r) > c_2(r,i_r),
\]
implying $c_1>c_2$, contradicting the originally defined order $c_1<c_2$.

\underline{Part 5.}
We have $\type(c_0,d)\le\type(c_2,d)$ by Part~1. Suppose, for the purpose of obtaining a contradiction, that
$\type(c_0,d) = \type(c_2,d)=r$, 
which implies, again by Part~1, that $\type(c_1,d)=r$ as well.
We have assumed $\type(a,b)\le r$ and $\type(a,b)=r$ would contradict Part~4, so we have $\type(a,b)<r$. 
With the notation introduced in Part 4, 
both pairs $(c_0,d)$ and $(c_2,d)$ first differ 
at coordinate $(r,i_r)$, while $(c_1,d)$ first differ at coordinate $(r,j_r)$. 
We have $c_0<c_1<c_2<d$ lexicographically, 
so we must also have $i_r=j_r$.
We further know that
\begin{align*}
d(r,i_r)-c_0(r,i_r)=d(r,i_r)-c_2(r,i_r)
&=\vect[i_1,\dots,i_t](r,i_r)=\langle i_1,\dots,i_{r-1}\rangle,\\
d(r,j_r)-c_1(r,j_r)&=\vect[j_1,\dots,j_t](r,j_r)=\langle j_1,\dots,j_{r-1}\rangle.
\intertext{This implies}
c_0(r,i_r) = c_2(r,i_r) &\neq c_1(r,i_r)
\end{align*}
since Part~2 and $\type(a,b)<r$ 
implies that $\langle i_1,\dots,i_{r-1}\rangle\ne\langle j_1,\dots,j_{r-1}\rangle$.
Now $(r,i_r)$ is the first coordinate where $c_1$ differs from either $c_0$ or $c_2$, but $c_0$ and $c_2$ agree on this coordinate,
which contradicts the ordering $c_0<c_1<c_2$. 
The confirms Part~5 of the lemma.
\end{proof}
Define the alternating patterns $P_t$ and $Q_t$,
where $Q_t$ is a reflection of $P_t$ across 
the minor diagonal.
\begin{align*}
    P_t &= \left(\begin{array}{ccccccc}
\zero{\hcm[-.2]\rb{4}{$\overbrace{\hcm[3.6]}^{\text{$2t-1$ alternating 1s}}$}}
\bu &&\bu&&&\bu&\bu\\
&\bu&&\rb{2.4}{$\cdots$}&\bu&&\bu
\end{array}\right)
&
Q_t &= \left(\begin{array}{cc}
\bu&\bu\\
&\bu\\
\bu&\\
\zero{\hcm[.25]$\vdots$}\\
&\bu\\
\bu&\\
&\bu
\end{array}\right)
\end{align*}

When $t\geq 2$, 
$P_{t}$ and $Q_{t}$ appear in $A_t$, 
and in fact $P_{t'},Q_{t'}$ appear in 
$A_t$ for \emph{every} constant $t'\geq t$.
Lemmas~\ref{lem:P} and \ref{lem:Q}
give useful constraints on \emph{how} 
$P_t,Q_t$ can be embedded in $A_t$.

\begin{lemma}\label{lem:P}
Consider an occurrence of 
$P_t$ in $A_t$, 
where $a,b\in \I$ are the indices of the two rows
and $c,d\in \I$ are the indices of the first and last columns.
\[
\begin{array}{c}
a\\
b
\end{array}
\left(\begin{array}{ccccccc}
\zero{\rb{3}{c}}\bu &&\bu&&&\bu&\zero{\rb{3}{d}}\bu\\
&\bu&&\rb{2.4}{$\cdots$}&\bu&&\bu
\end{array}\right)
\istrut{7}
\]
If $\type(a,b)\le\type(c,d)$, then $\type(a,b)=\type(c,d)=1$.
\end{lemma}

\begin{proof}
For $i\in[t]$, let $c_i\in\I$ be 
the index of the $(2i-1)^{\text{th}}$ column in this occurrence of $A_t$, so $c=c_1<c_2<\cdots<c_t<d$. We have
$\type(a,b)\le\type(c_1,d)\leq\cdots \leq \type(c_t,d)$, where the first inequality is assumed and rest follow from Lemma~\ref{lem:properties}(1). 
Part 5 of Lemma~\ref{lem:properties} applies and implies that these latter inequalities are \emph{strict}, i.e.,
\[
\type(a,b)\leq \type(c_1,d) < \type(c_2,d) < \cdots < \type(c_t,d).
\]
All types are from $[t]$, therefore both $\type(a,b)$ and $\type(c,d)=\type(c_1,d)$ must be 1.
\end{proof}

\begin{lemma}\label{lem:Q}
Consider an occurrence of $Q_t$ in $A_t$, 
where $a,b\in \I$ are the indices of the first and last rows and $c,d\in \I$ are the indices of the two columns. 
If $\type(a,b)\ge\type(c,d)$, then $\type(a,b)=\type(c,d)=1$.
\end{lemma}

\begin{proof}
First one has to establish the analogues of Lemma~\ref{lem:properties}(4,5) for patterns that are reflected across the minor diagonal, depicted below, 
then follow the proof of Lemma~\ref{lem:P}. 
\[
\left(\begin{array}{cc}
\bu&\bu\\
    &\bu\\
\bu&
\end{array}\right)
\hcm[2]
\left(\begin{array}{cc}
\bu&\bu\\
    &\bu\\
\bu&\\
    &\bu
\end{array}\right)
\]
For both parts, the original proofs work \emph{mutatis mutandis}.
\end{proof}

Define $X_t$ to be the following $2t\times (2t+1)$ pattern.
\[
X_t = \left(\begin{array}{ccccccccc}
&\bu &&\bu&\cdots&\bu&&\bu&\bu\\
&&&&&&&&\bu\\
&\bu&&&&&&&\\
&&&&&&&&\bu\\
&\bu&&&&&&&\vdots\\
&\vdots&&&&&&&\bu\\
&\bu&&&&&&&\\
\bu&&\bu&\cdots&\bu&&\bu&&\bu
\end{array}\right)
\]

Alternatively, $X_t$ is defined to be the 0--1 matrix whose first and last rows with the first column removed form $P_t$, while its second and last column form $Q_t$ and has a single 1 entry outside these submatrices in the first column and last row.

\begin{lemma}\label{lem:avoids-X}
$A_t$ avoids $X_t$.
\end{lemma}

\begin{proof}
    Suppose there is an occurrence of $X_t$ in $A_t$.  
    Let $a,b\in\I$ be the indices of the first 
    and last rows of the $X_t$ instance,
    and let $c',c,d\in\I$ be the indices of its first, second 
    and last columns. We either have $\type(a,b)\le\type(c,d)$ or $\type(a,b)\ge\type(c,d)$. We must have $\type(a,b)=\type(c,d)=1$ in both cases by Lemmas~\ref{lem:P} and \ref{lem:Q}, respectively. But then $\type(c',d)$ is also 1 by Lemma~\ref{lem:properties}(1) and then the rows $a,b$ and columns $c',c,d$ contradict Lemma~\ref{lem:properties}(4). The contradiction proves our lemma.
\end{proof}

Lemma~\ref{lem:avoids-X} lets us obtain a lower bound on 
$\Ex(X_t,n)$.  We use a Lemma of Pach and Tardos~\cite{PachTardos06} to get a nearly matching upper bound.

\begin{lemma}[Lemmas~3 of~\cite{PachTardos06}]\label{lem:PachTardos}
Let $P$ be a 0--1 pattern with rows $i_0,i_1$ and a column $j$ such that $P(i_0,j)=1$ is the only
1 in column $j$, and 
$P(i_0,j+1)=P(i_1,j-1)=P(i_1,j+1)=1$.
\smallskip

\[
P = 
\scalebox{.8}{$\begin{array}{r}
\ \\
i_0\\
\ \\
i_1\\
\ 
\end{array}
\left(\begin{array}{@{\hcm[1]}ccc@{\hcm[1]}}
&&\\
\zero{\hcm[-.5]\rb{8}{$j-1$}}{}&\zero{\hcm[.1]\rb{8}{$j$}}\bu&\zero{\hcm[-.05]\rb{8}{$j+1$}}\bu\\
& & \\
\bu& &\bu \\
&&
\end{array}
\right)$}
\]
Let $P'$ be $P$ with column $j$ removed. Then
$\Ex(P,n) = O(\Ex(P',n)\cdot\log n)$.
\end{lemma}

Since the extremal function of a pattern is invariant under
rotations and reflections, Lemma~\ref{lem:PachTardos} also
applies with the roles of rows and columns reversed, and the the roles of 
$j-1$ and $j+1$ reversed.

\medskip 

\begin{proof}[Proof of Theorem~\ref{thm:main}]
By Lemmas~\ref{lem:density} and \ref{lem:avoids-X}, 
$A_t$ has weight $\Theta(n(\log n/\log\log n)^t)$
and avoids $X_t$, giving the lower bound.
For the upper bound, 
we can apply Lemma~\ref{lem:PachTardos} 
iteratively to remove rows $2,3,\dots,2t-1$ followed by
columns $2t,2t-1,\dots,2$, leaving a linear pattern with three 1s.
Each application of Lemma~\ref{lem:PachTardos} introduces 
a $\log n$ factor, 
so $\Ex(X_t,n)=O(n\log^{\|X_t\|_1-3}n)=O(n\log^{4t-3} n)$.
\end{proof}

\section{Conclusion and Open Problems}\label{sect:conclusion}

The following broad classification of 0--1 patterns
comes out of the last 30+ years of 
forbidden 0--1 matrix theory~\cite{BienstockG91,FurediH92,Klazar92,KV94,MarcusT04,Tardos05,PachTardos06,Keszegh09,Fulek09,Geneson09,Pettie-FH11,Pettie-GenDS11,Pettie-SoCG11,Timmons12,Fox13,Pettie15-SIDMA,CibulkaK17,WellmanP18,GyoriKMTTV18,KorandiTTW19,FurediKMV20,MethukuT22,GerbnerMNPTV23,KucheriyaT23,KucheriyaT23b}.

\begin{description}
    \item[Linear Patterns.] 
        Let $\mathcal{P}_{\operatorname{lin}}$ be the set of all $P$ such that $\Ex(P,n)=O(n)$. $\mathcal{P}_{\operatorname{lin}}$ 
        contains several well-structured classes of patterns such as permutations~\cite{MarcusT04}, ``double'' permutations~\cite{Geneson09}, and monotone patterns~\cite{Pettie-SoCG11,Keszegh09,KV94}. 
        Examples of the last two classes are
        \begin{align*}
            \scalebox{.6}{$\left(
            \begin{array}{cccccccc}
            &&&&&&\bu&\bu\\
            &&\bu&\bu&&&&\\
            &&&&\bu&\bu&&\\
            \bu&\bu&&&&&&
            \end{array}
            \right)$}
            \hcm[2]
            \scalebox{.6}{$\left(
            \begin{array}{cccccccccccccc}
            &&&&&&\bu&\bu&&&&&\\
            &&&&\bu&\bu&&&\bu&\bu&&&\\
            &&\bu&\bu&&&&&&&\bu&\bu&&\\
            \bu&\bu&&&&&&&&&&&\bu&\bu\\
            \end{array}
            \right)$}
        \end{align*}
        There are only a handful of known linear patterns outside these classes.
        The first two examples 
        below are proved via ad hoc arguments~\cite{Fulek09,Pettie-GenDS11},
        and the last is an example of the ``grafting'' operation~\cite{Pettie-SoCG11} applied to a linear pattern~\cite{Tardos05}.
        \begin{align*}
            \scalebox{.6}{$\left(
            \begin{array}{ccccccc}
            &&&&&\bu&\\
            &\bu&\bu&&&&\\
            &&&\bu&\bu&&\bu\\
            \bu&&&&
            \end{array}
            \right)$}
\hcm[1]
             \scalebox{.6}{$\left(
            \begin{array}{ccccc}
            &&\bu&&\\
            \bu&&&&\bu\\
            &\bu&\bu&\bu&
            \end{array}
            \right)$}
\hcm[1]
            \scalebox{.6}{$\left(\begin{array}{ccccccc}
            &&&\bu\\
            &&\bu\\
            &&\bu\\
            \bu\\
            &\bu&\bu&&\bu&\bu\\
            &&&&&&\bu\\
            &&\bu
\end{array}\right)$}
        \end{align*}
        A difficult open problem is to characterize the class $\mathcal{P}_{\operatorname{lin}}$.  It is known that 
        there are infinitely many minimal\footnote{with respect to $\prec$ containment}
        nonlinear patterns~\cite{Keszegh09,Geneson09,Pettie-FH11}.
        On the other hand, every known $P$ for which 
        $P\not\in \mathcal{P}_{\operatorname{lin}}$
        is witnessed by one of two constructions~\cite{HS86,FurediH92} 
        with weight $\Theta(n\alpha(n))$ and $\Theta(n\log n)$ (where $\alpha(n)$ is the inverse-Ackermann function),
        and every $P$ with $\Ex(P,n)=\Omega(n\log n)$ is
        witnessed by one of two closely related constructions
        \cite{FurediH92,Tardos05}.  It may be that a finite number of witness constructions characterize the set of patterns outside of $\mathcal{P}_{\operatorname{lin}}$.

        Although characterizing linear patterns seems to be beyond reach, 
        a nice and simple characterization of linear \emph{connected} patterns is given in \cite{FurediKMV20}. 
        Here we call $P$ \emph{connected} if the corresponding bipartite graph $G(P)$ is connected.
\item[Quasilinear Patterns.] Let $\mathcal{P}_{\operatorname{qlin}}$ be the set
of all $P$ for which $\Ex(P,n)\leq n2^{(\alpha(n))^{O(1)}}$,
where $\alpha(n)$ is the inverse-Ackermann function.
Functions of this type are called \emph{quasilinear} 
and show up in the analysis of (generalized) Davenport-Schinzel sequences~\cite{HS86,ASS89,Klazar92,Nivasch10,Pettie-DS-JACM,Pettie-GenDS11,Pettie15-SIDMA} and other combinatorial problems~\cite{AlonKNSS08-J}.
A pattern is \emph{light}
if it contains exactly one 1 per column.  
All light patterns are in $\mathcal{P}_{\operatorname{qlin}}$~\cite{Klazar92,Nivasch10,Pettie15-SIDMA} and all patterns known to be in $\mathcal{P}_{\operatorname{qlin}}\backslash \mathcal{P}_{\operatorname{lin}}$ are either light,
or composed of light or linear patterns via Keszegh's~\cite{Keszegh09} joining operation.\footnote{Keszegh~\cite{Keszegh09} observed that if $A$ has a 1 in its southeast corner and $B$ has a 1 in its notherwest corner, that by \emph{joining} them at their corners, the resulting pattern $A\oplus B$ has extremal function $\Ex(A\oplus B,n) \leq \Ex(A,n)+\Ex(B,n)$.
\[
A\oplus B = 
\left(\,\begin{array}{ccccccc}
\cline{1-4}
\multicolumn{1}{|c}{ } &&& \multicolumn{1}{c|}{} \\
\multicolumn{1}{|c}{ } &A&& \multicolumn{1}{c|}{} &&& \\\cline{4-7}
\multicolumn{1}{|c}{ } &&\zero{\hcm[.26]{\bu}}& \multicolumn{1}{|c|}{} &&& \multicolumn{1}{c|}{}\\\cline{1-4}
&&& \multicolumn{1}{|c}{} &&B& \multicolumn{1}{c|}{}\\
&&& \multicolumn{1}{|c}{} &&& \multicolumn{1}{c|}{}\\\cline{4-7}
\end{array}\,\right)
\]
If $A$ is light, and $B$ is the transpose of a light pattern, then
$A\oplus B\in \mathcal{P}_{\operatorname{qlin}}$, but neither it nor 
its transpose is light.}
It is an open question whether there are infinitely many \emph{minimal}
non-linear patterns in $\mathcal{P}_{\operatorname{qlin}}$.  
It is consistent with known results that if $P$ is light, 
then $P\in \mathcal{P}_{\operatorname{lin}}$ iff it 
avoids the following patterns (or their reflections), which
correspond to order-3 Davenport-Schinzel sequences.
\begin{align*}
    \scalebox{.6}{$\left(\begin{array}{cccc}
    \bu&&&\\
    &&\bu&\\
    &\bu&&\bu
    \end{array}\right)$}
\hcm[1]
    \scalebox{.6}{$\left(\begin{array}{cccc}
    \bu&&\bu&\\
    &\bu&&\bu
    \end{array}\right)$}
\end{align*}
It would also be of interest to characterize, for each $t\geq 1$, 
(light) patterns $P\in\mathcal{P}_{\operatorname{qlin}}$ 
for which $\Ex(P,n)\geq n2^{\Omega(\alpha^t(n))}$.
\item[Acyclic Patterns.] All acyclic $P$ for which 
the bound $\Ex(P,n) \leq n(\log n)^{O(1)}$ is known 
can be proved via
the Pach-Tardos reductions~\cite[Lemmas 2, 3, and 4]{PachTardos06}  (see Lemma~\ref{lem:PachTardos} for one),
together with Keszegh's~\cite{Keszegh09} joining operation.\footnote{For example,
the following pattern has extremal function $O(n\log^3 n)$, but it is not subject
to any of the Pach-Tardos reductions.  It must first be decomposed into two 
patterns via Keszegh~\cite{Keszegh09}, 
each of which has extremal function $O(n\log^3 n)$ by~\cite{PachTardos06}.
\[
\left(\, \begin{array}{cccccc}
\cline{1-3}
\multicolumn{1}{|c}{\bu} & \bu & \multicolumn{1}{c|}{}\\
\multicolumn{1}{|c}{\bu} && \multicolumn{1}{c|}{\bu}\\
\multicolumn{1}{|c}{} &\bu& \multicolumn{1}{c|}{}\\\cline{3-6}
\multicolumn{1}{|c}{} &&\multicolumn{1}{|c|}{\bu}&&\bu&\multicolumn{1}{c|}{}\\\cline{1-3}
&&\multicolumn{1}{|c}{}&\bu&&\multicolumn{1}{c|}{\bu}\\
&&\multicolumn{1}{|c}{}&&\bu&\multicolumn{1}{c|}{\bu}\\\cline{3-6}
\end{array}\, \right)
\]}
If $P$ is \emph{degenerate}
then $\Ex(P,n)\leq n^{1+o(1)}$~\cite{KorandiTTW19}.
Upper bounding $\Ex(P,n)$ by $n(\log n)^{O_P(1)}$,
$n^{1+o(1)}$, or even $n^{2-\epsilon}$ for \emph{all} acyclic patterns $P$ is the main open problem in this area.  A more subtle problem is to determine which extremal functions are possible.  Tardos~\cite{Tardos05} gave examples
of \emph{pairs} of patterns with $\Ex(\{P,P'\},n)=\Theta(n\log\log n)$ and $\Ex(\{P'',P'''\},n)=\Theta(n\log n/\log\log n)$, but it is unknown whether these extremal functions can be achieved by 
a \emph{single} forbidden pattern.
\item[Arbitrary Patterns.] 
The \Kovari-\Sos-\Turan{} theorem~\cite{KovariST54}
implies that if $P\in\{0,1\}^{k\times l}$, then $\Ex(P,n) = O\left(n^{2-\frac{1}{\min\{k,l\}}}\right)$.
Pach and Tardos~\cite{PachTardos06} constructed $\Theta(n^{4/3})$-weight
matrices that avoid some arbitrarily long ordered cycles.
See Timmons~\cite{Timmons12} and \Gyori{} et al.~\cite{GyoriKMTTV18}
for more results on ordered cycles.  
Methuku and Tomon~\cite{MethukuT22} defined a matrix $P$ to be 
\emph{row $t$-partite} if it can be cut along rows into $t$ light matrices,
and $t\times t$-partite if both $P$ and $P^T$ are row $t$-partite.
They proved that if $P$ is row $t$-partite and $t\times t$-partite, that
$\Ex(P,n)$ is at most $n^{2-1/t+1/t^2 +o(1)}$ and $n^{2-1/t+o(1)}$, respectively.
\end{description}

A 0--1 matrix can be viewed as an 
\emph{ordered bipartite} graph,
where the two parts of the bipartition 
are given independent linear orders.  
Forbidden 0--1 matrix theory has been extended
to other types of ordered subgraph containment.

\begin{description}
\item[Vertex-ordered graphs] It is natural to drop the requirement that the forbidden graph and host graph be bipartite and just consider the extremal theory of 
arbitrary \emph{vertex-ordered graphs}: these are simple graphs with a linear order on their vertices. Containment between vertex-ordered graphs must preserve the ordering. The extremal function of such a (forbidden) graph $H$ was introduced in \cite{PachTardos06}: $\Ex(H,n)$ is the maximum number of edges of an $n$-vertex vertex-ordered graph that does not contain $H$. The connection to the extremal theory is 0--1 patterns is very close. A vertex-ordered graphs $H$ is called \emph{ordered bipartite} (or of interval chromatic number 2) if it is a bipartite graph with one partite class of vertices preceding the other in the vertex-order. If $H$ is not ordered bipartite, then $\Ex(H,n)=\Theta(n^2)$. If $H$ is ordered bipartite, let $P(H)$ be its bipartite 0--1 adjacency matrix (ordering the rows and columns consistent with the given vertex-order) 
and we have
$$\Ex(P(H),n/2)\le\Ex(H,n)=O(\Ex(P(H),n)\cdot \log n).$$
This implies that the vertex-ordered graph $H$ for which
$P(H)=X_t$ is an example of an ordered bipartite tree whose extremal function is $\Omega(n(\log n/\log\log n)^t)$. Previously no ordered bipartite tree was known whose extremal function was not $n\log^{1+o(1)}n$.

Although the connection between the 0--1 matrices and vertex-ordered graph is not close enough to directly translate questions about the the \emph{linearity} 
of extremal functions, the situation was similar in the two theories: although characterization of forbidden vertex-ordered graphs $H$ with $\Ex(H,n)=O(n)$ is currently beyond reach, the paper \cite{FurediKMV20} provides such a 
characterization for \emph{connected} $H$.
\item[Edge-ordered graphs] Rather than extend \Turan-type extremal graph theory by adding a total order on \emph{vertices},
we could instead add a total order on \emph{edges}.
This yields the extremal theory of \emph{edge-ordered graphs} as introduced by \cite{GerbnerMNPTV23}. A rich theory starts to form but it is not as closely related to the extremal theory of 0--1 patterns as the vertex-ordered variant is. Nevertheless, many results and problems have analogues in the two theories. The analogue of the interval chromatic number is the \emph{order chromatic number}: an edge-ordered graph has order chromatic number 2 if it is contained in the lexicographically ordered complete bipartite graph. The extremal function of an edge-ordered graph is $\Theta(n^2)$ if and only if its order chromatic number is not 2. While the characterization of edge-ordered graphs with linear extremal functions seems to also be beyond reach in general, 
such a characterization is given for connected edge-ordered graphs~\cite{KucheriyaT23b}. 
The extremal function of acyclic edge-ordered graphs of order-chromatic number 2 was conjectured to be $n^{1+o(1)}$ in \cite{GerbnerMNPTV23} and this has recently been established in \cite{KucheriyaT23}, where the stronger (and still open) conjecture was formulated that these extremal functions are all of the form $O(n\log^c n)$, where $c$ may depend on the forbidden edge-ordered graph. 
In contrast to the main result of this paper, it is still possible that the above bound holds with $c=1$ for \emph{all} acyclic edge-ordered graphs of order chromatic number 2.
\end{description}

\newcommand{\etalchar}[1]{$^{#1}$}

\end{document}